\documentclass[12pt,reqno]{amsart}
\usepackage{amscd,amsmath,amsthm,amssymb}
\usepackage{color}
\usepackage{tikz}
\usepackage{url}
\usepackage{circuitikz}
\usepackage{float}

\newtheorem{Theorem}{Theorem}[section]
\newtheorem{Lemma}[Theorem]{Lemma}
\newtheorem{Corollary}[Theorem]{Corollary}

\newtheorem{Remark}[Theorem]{Remark}

\newtheorem{Question}[Theorem]{Question}

\newtheorem{TheoremAlph}{Theorem}

\def\qed{\ifhmode\textqed\fi
	\ifmmode\ifinner\quad\qedsymbol\else\dispqed\fi\fi}
\def\textqed{\unskip\nobreak\penalty50
	\hskip2em\hbox{}\nobreak\hfill\qedsymbol
	\parfillskip=0pt \finalhyphendemerits=0}
\def\dispqed{\rlap{\qquad\qedsymbol}}

\def\Ker{\textup{Ker}}

\def\m{\mathfrak{m}}

\def\FF{\mathbb{F}}
\def\Tor{\textup{Tor}}

\def\lcm{\textup{lcm}}

\def\supp{\textup{supp}}
\def\reg{\textup{reg}}

\def\lex{\textup{lex}}

\def\ini{\textup{in}}

\begin{document}
	
	\title{Stanley-Reisner ideals with linear powers}
	\author{Antonino Ficarra, Somayeh Moradi}
	
	\address{Antonino Ficarra, BCAM -- Basque Center for Applied Mathematics, Mazarredo 14, 48009 Bilbao, Basque Country -- Spain, Ikerbasque, Basque Foundation for Science, Plaza Euskadi 5, 48009 Bilbao, Basque Country -- Spain}
	\email{aficarra@bcamath.org,\,\,\,\,\,\,\,\,\,\,\,\,\,antficarra@unime.it}
	
	\address{Somayeh Moradi, Department of Mathematics, Faculty of Science, Ilam University, P.O.Box 69315-516, Ilam, Iran}
	\email{so.moradi@ilam.ac.ir}
	
	
	\begin{abstract}
		Let $S = K[x_1, \dots, x_n]$ be the standard graded polynomial ring over a field $K$. In this paper, we address and completely solve two fundamental open questions in Commutative Algebra:\medskip 
		\begin{enumerate}
			\item[\textup{(i)}]
			For which degrees $d$, does there exist a uniform combinatorial characterization of all squarefree monomial ideals in $S$ having $d$-linear resolutions?\smallskip
			\item[(ii)] For which degrees $d$, does having a linear resolution coincide with having linear powers for all squarefree monomial ideals of $S$ generated in degree $d$?
		\end{enumerate}\medskip
		Let $\mathcal{I}_{n,d}(K)$ denote the class of squarefree monomial ideals of $S$ having a $d$-linear resolution. Our main result establishes the equivalence of the following conditions:\medskip
		\begin{enumerate}
			\item[\textup{(a)}] Any squarefree monomial ideal $I$ in $S$ generated in degree $d$ has a linear resolution, if and only if, $I$ has linear powers.\smallskip
			\item[\textup{(b)}] $\mathcal{I}_{n,d}(K)$ is independent of the base field $K$.\smallskip
			\item[\textup{(c)}] $d\in\{0,1,2,n{-}2,n{-}1,n\}$.
		\end{enumerate}\medskip
		In each of these degrees, we show that a squarefree monomial ideal has a linear resolution if and only if all of its powers admit linear quotients, and we combinatorially classify such ideals. In contrast, for each degree $3\le d\le n{-}3$, we construct fully-supported squarefree monomial ideals $I$ and $J$ in $S$ generated in degree $d$ such that the linear resolution property of $I$ depends on the choice of the base field, $J$ has a linear resolution and $J^2$ does not have a linear resolution. 
	\end{abstract}
	
	\maketitle
	\vspace*{-1.8em}
	\section{Introduction}
	
	Let $S = K[x_1, \dots, x_n]$ be the standard graded polynomial ring over a field $K$, and let $I \subset S$ be a squarefree monomial ideal. The foundational works of Hochster, Stanley, and Reisner revealed a deep correspondence between such ideals and the combinatorics of simplicial complexes: $I$ is equal to the Stanley-Reisner ideal $I_\Delta$ of a unique simplicial complex $\Delta$ on the vertex set $[n]=\{1,\dots,n\}$. This algebra-topology correspondence has proved remarkably fruitful, enabling Stanley to resolve the Upper Bound Conjecture for simplicial spheres (cf. \cite[Section 5.4]{BH}).
	
	A central invariant in the homological study of graded ideals is the Castelnuovo-Mumford regularity, originally introduced by Eisenbud-Goto and, independently, by Ooishi. For a graded ideal $I \subset S$, it is defined as
	$$
	\reg\,I\ =\ \max\{j+i:\ \textup{H}_\m^i(I)_j\ne0\},
	$$
	where $\textup{H}_\m^i(I)_j$ is the $j$th graded piece of the $i$th local cohomology module of $I$ with respect to the (unique) graded maximal ideal $\m=(x_1,\dots,x_n)$ of $S$. One always has $\reg\,I\ge\alpha(I)$, where $\alpha(I)$ is the minimal degree among the homogeneous generators of $I$, and equality precisely characterizes those ideals with a linear resolution.
	
	A natural refinement of this notion arises in  the study of powers of ideals. If $I^k$ has a linear resolution for all $k\ge1$, we say that $I$ has {\em linear powers}. While it may be tempting to conjecture that the linear resolution of $I$ implies linear powers, counterexamples due to Terai and Sturmfels reveal that this implication fails in general and, in some cases, depends on the base field $K$.
	
	Terai showed that if the characteristic of $K$ is not $2$, the Stanley-Reisner ideal
	\begin{equation}\label{eq:Te}
		J\ =\ (abd,\,abf,\,ace,\,adc,\,aef,\,bde,\,bcf,\,bce,\,cdf,\,def)
	\end{equation}
	of the minimal triangulation of the projective plane has a linear resolution, while $J^2$ does not have a linear resolution. This example is characteristic-dependent.
	
	Sturmfels \cite{S} provided another example of an ideal
	\begin{equation}\label{eq:St}
		J\ =\ (def,\,cef,\,cdf,\,cde,\,bef,\,bcd,\,acf,\,ade)
	\end{equation}
	with a linear resolution, indeed even linear quotients (a concept which we recall below), for which $J^2$ does not have a linear resolution. This example does not depend on $K$.
	
	For a monomial ideal $I\subset S$, let $\mathcal{G}(I)$ be the (unique) minimal monomial generating set of $I$. We say that $I$ has {\em linear quotients} if there exists an order $u_1,\dots,u_m$ on $\mathcal{G}(I)$ such that $(u_1,\dots,u_{j-1}):(u_j)$ is generated by variables, for all $j=2,\dots,m$. If $I$ is generated in a single degree and has linear quotients, then it has a linear resolution, regardless of the choice of the field $K$.
	
	For an integer $0\le d\le n$, we set
	$$
	\mathcal{I}_{n,d}(K)=\{I\subset S:\ I\ \textup{is a squarefree monomial ideal with a}\ d\textup{-linear resolution}\}.
	$$
	
	In the light of the above discussions and examples (\ref{eq:Te})-(\ref{eq:St}), the following fundamental questions have long remained elusive and unsolved.
	\begin{Question}\label{Question}
		When is $\mathcal{I}_{n,d}(K)$ independent of $K$?
	\end{Question}
	\begin{Question}\label{Question1}
		For which degrees $d$, does every squarefree monomial ideal $I\subset S$ generated in degree $d$ have a linear resolution, if and only if, it has linear powers?
	\end{Question}
	
	That is, for which integers $n$ and $d$, does there exist a uniform combinatorial characterization of all squarefree monomial ideals generated in degree $d$, that have linear resolutions or even linear powers? And for which degrees $d$, having a linear resolution and having linear powers are equivalent properties for all squarefree monomial ideals of $S$ generated in degree $d$? In this paper, we will solve these problems.
	
	Any squarefree monomial ideal $I\subset S$ generated in degree $d\in\{0,1,n{-}1,n\}$ is polymatroidal. A monomial ideal $I\subset S$ is called {\em polymatroidal} if the exponent vectors of the monomials of $\mathcal{G}(I)$ form the set of bases of a discrete polymatroid (cf \cite[Chapter 12]{HHBook}). By a result of Bandari and Rahmati-Asghar \cite{BR}, a monomial ideal $I\subset S$ is polymatroidal, if and only if, $I$ is generated in a single degree and it has linear quotients with respect to the lexicographic order induced by any ordering of the variables. Since all powers of a polymatroidal ideal are again polymatroidal, the linear resolution property of any power of $I$ is independent from the base field $K$.
	
	In view of examples (\ref{eq:Te}) and (\ref{eq:St}), the first interesting case occurs when $d=2$. Then $I$ can be viewed as an edge ideal $$I(G)\ =\ (x_ix_j:\ \{i,j\}\in E(G))$$ of a finite simple graph $G$ on vertex set $V(G)=[n]$ with edge set $E(G)$. Recall that the complementary graph of $G$ is the graph $G^c$ with $V(G^c)=V(G)$ whose edges are the non-edges of $G$. Herzog, Hibi and Zheng proved the following influential result, see \cite{HHZ} and \cite[Corollary 3.2]{HHDepth}. See also \cite{BDMS} and \cite{FShort}, for a short proof. 
	
	\begin{TheoremAlph}\label{ThmA}
		Let $I\subset S$ be a squarefree monomial ideal generated in degree $2$. Write $I=I(G)$. The following conditions are equivalent.
		\begin{enumerate}
			\item[\textup{(a)}] $I$ has a linear resolution.
			\item[\textup{(b)}] $I^k$ has linear quotients, for all $k\ge1$.
			\item[\textup{(c)}] $G^c$ is chordal.
		\end{enumerate}
	\end{TheoremAlph}
	
	For squarefree quadratic monomial ideals, the bad behavior of example (\ref{eq:St}) does not occur: if $I\subset S$ has a $2$-linear resolution, then $I^k$ has a $2k$-linear resolution, for all $k\ge1$. The equivalence (a) $\Leftrightarrow$ (b) in Theorem \ref{ThmA} holds as well for the ideals $I\in\mathcal{I}_{n,d}(K)$ with $d\in\{0,1,n{-}1,n\}$, because any such ideal is polymatroidal and powers of polymatroidal ideals are polymatroidal \cite[Theorem 12.6.3]{HHBook}.
	
	\smallskip
	The ideals (\ref{eq:Te}), (\ref{eq:St}) are generated in degree $3$ and live in a polynomial ring in $6$ variables. This fact suggests that Theorem \ref{ThmA} could be extended in degree $d=n{-}2$.
	
	Let $I\subset S$ be a squarefree monomial ideal generated in degree $n{-}2$. Given a non-empty subset $F$ of $[n]$, we put ${\bf x}_F=\prod_{i\in F}x_i$. Then $\mathcal{G}(I)=\{{\bf x}_{[n]\setminus e_1},\dots,{\bf x}_{[n]\setminus e_m}\}$ where $\{e_1,\dots,e_m\}$ is the edge set of a finite simple graph $G$ on $[n]$. Conversely, given a finite simple graph $G$, in \cite{FM1} we introduce the {\em complementary edge ideal} of $G$ as
	$$
	I_c(G)\ =\ ({\bf x}_{[n]\setminus e}:\ e\in E(G)).
	$$
	
	Any squarefree monomial ideal generated in degree $n{-}2$ is the complementary edge ideal of some graph $G$. We denote by $c(G)$ the number of connected components of $G$ having more than one vertex. We have the following extension of Theorem \ref{ThmA}.
	\begin{TheoremAlph}\label{ThmB}
		Let $I\subset S$ be a squarefree monomial ideal generated in degree $n{-}2$. Write $I=I_c(G)$. The following conditions are equivalent.
		\begin{enumerate}
			\item[\textup{(a)}] $I$ has a linear resolution.
			\item[\textup{(b)}] $I^k$ has linear quotients, for all $k\ge1$.
			\item[\textup{(c)}] $c(G)=1$.
		\end{enumerate}
	\end{TheoremAlph}
	
	In order to prove Theorem \ref{ThmB}, we follow the strategy outlined in \cite{HHZ}. Let $I\subset S$ be a monomial ideal generated in a single degree with $\mathcal{G}(I)=\{u_1,\dots,u_m\}$. The Rees algebra of $I$ is defined as the $K$-subalgebra $\mathcal{R}(I)=\bigoplus_{k\ge0}I^kt^k$ of $S[t]$. Let
	$$
	\varphi:\ T=S[y_1,\dots,y_m]\rightarrow\mathcal{R}(I)
	$$
	be the $S$-algebra homomorphism defined by $\varphi(y_j)=u_jt$, for all $j=1,\dots,m$, where $\deg_T(x_i)=(1,0)$ and $\deg_T(y_j)=(\deg_S(u_j),1)$. Then $\mathcal{R}(I)\cong T/J$, where $J=\Ker\,\varphi$ is a toric ideal. The \textit{$x$-regularity} of $\mathcal{R}(I)$ is defined as
	$$
	\reg_x\mathcal{R}(I)\ =\ \max_{i,j}\{a_{i,j}-i\},
	$$
	where $$\FF\ :\ 0\rightarrow F_p\rightarrow F_{p-1}\rightarrow\cdots\rightarrow F_0\rightarrow\mathcal{R}(I)\rightarrow0$$ is the minimal bigraded free $T$-resolution of $\mathcal{R}(I)$ with $F_i=\bigoplus_jT(-a_{i,j},-b_{i,j})$ for all $i$. The following results were proved in \cite{Ro}, \cite{BC} and \cite{HHZ}, respectively.
	
	\begin{Theorem}\label{Thm:regx}
		Let $I\subset S$ be a graded ideal generated in degree $d$. The following statements hold.
		\begin{enumerate}
			\item[\textup{(a)}] $\reg\,I^k\le dk+\reg_x\mathcal{R}(I)$ for all $k\ge1$.
			\item[\textup{(b)}] $I$ has linear powers, if and only if, $\reg_x\mathcal{R}(I)=0$.
			\item[\textup{(c)}] \textup{(\textbf{The $x$-condition})} If there exists a monomial order $<$ on $T$ such that $\ini_<(J)$ is generated by monomials of the form $uv$ with $u\in S$ of degree $\le1$ and $v\in K[y_1,\dots,y_m]$, then $\reg_x\mathcal{R}(I)=0$.
		\end{enumerate}
	\end{Theorem}\medskip
	
	To establish Theorem \ref{ThmB}, we apply Theorem \ref{Thm:regx}(c). By itself, condition (c), which is usually called the {\em $x$-condition} (see \cite{HHM,HHZ,HM}), only guarantees that $I$ has linear powers. In Section \ref{sec2}, we introduce the concepts of $x$-dominant and $y$-dominant orders, and we prove that if condition (c) is satisfied with respect to such monomial orders, then $I$ has not only linear powers, but also linear quotients, for all $k\ge1$ (see Theorems \ref{GeneralStatementX} and \ref{GeneralStatementY}). The order used in the proof of Theorem \ref{LinearPowers} is $x$-dominant and this together with Corollary \ref{Cord=n-2} establishes Theorem \ref{ThmB}.
	
	\medskip
	In order to fully address Questions \ref{Question} and \ref{Question1}, it remains to consider the degrees $3\le d\le n{-}3$. Notice that for such a $d$ to exist we must have $n\ge6$. When $n=6$, the ideal (\ref{eq:Te}) shows that the set $\mathcal{I}_{6,3}(K)$ depends on the ground field $K$.\smallskip
	
	Recall that the support of a monomial $u\in S$ is the set $\supp(u)=\{i:x_i\ \textup{divides}\ u\}$. The support of a monomial ideal $I\subset S$ is defined as $\supp(I)=\bigcup_{u\in\mathcal{G}(I)}\supp(u).$ We say that $I\subset S$ is {\em fully-supported} if $\supp(I)=[n]$.\medskip
	
	Generalizing the ideal in (\ref{eq:Te}), we show that 
	
	\begin{TheoremAlph}\label{ThmC}
		Let $n\ge6$. For each $3\le d\le n{-}3$, there exists a fully-supported squarefree monomial ideal $I\subset S$ generated in degree $d$ such that
		$$
		\reg\,I\ =\ \begin{cases}
			d+1&\textit{if}\ \textup{char}(K)=2,\\
			\hfil d&\textit{otherwise}.
		\end{cases}
		$$
	\end{TheoremAlph}
	
	Hence, for each degree $3\le d\le n{-}3$ the set $\mathcal{I}_{n,d}(K)$ actually depends on $K$.\medskip
	
	Finally, it remains to address Question \ref{Question1} for each degree $3\le d\le n{-}3$. Notice that for such a $d$ to exist we must have $n\ge6$. When $n=6$, the ideal (\ref{eq:St}) shows that $\mathcal{I}_{6,3}(K)\ne\mathcal{I}_{6,3}^*(K)$ for any field $K$.\smallskip
	
	In the following result, we show that except for the degrees $d\in\{0,1,2,n{-}2,n{-}1,n\}$, if $I\in\mathcal{I}_{n,d}(K)$ is fully-supported and has linear quotients, it is not necessarily true that $I^k$ has linear resolution for all $k\ge1$.\medskip
	
	Generalizing the ideal in (\ref{eq:St}), we prove
	
	\begin{TheoremAlph}\label{ThmD}
		Let $n\ge6$. For each $3\le d\le n{-}3$, there exists a fully-supported squarefree monomial ideal $I\subset S$ generated in degree $d$ such that $I$ has linear quotients and $\reg\,I^2=2d+1$. Hence, $I^2$ does not have a linear resolution.
	\end{TheoremAlph}
	
	We define the class $$\mathcal{I}_{n,d}^*(K)\ =\ \{I\in\mathcal{I}_{n,d}(K):I\ \textup{has linear powers}\}.$$
	Theorem \ref{ThmD} shows that $\mathcal{I}_{n,d}(K)\ne\mathcal{I}_{n,d}^*(K)$ for any $3\le d\le n{-}3$ and any field $K$.\medskip
	
	By combining Theorems \ref{ThmA}, \ref{ThmB}, \ref{ThmC} and \ref{ThmD}, we are finally able to deliver the complete solution to Questions \ref{Question} and \ref{Question1}.
	
	\begin{TheoremAlph}\label{ThmE}
		The following statements are equivalent.
		\begin{enumerate}
			\item[\textup{(a)}] $\mathcal{I}_{n,d}(K)$ does not depend on $K$.
			\item[\textup{(b)}] $d\in\{0,1,2,n{-}2,n{-}1,n\}$.
			\item[\textup{(c)}] $\mathcal{I}_{n,d}(K)=\mathcal{I}_{n,d}^*(K)$.
		\end{enumerate}
		Furthermore, for any squarefree monomial ideal $I\subset S$ generated in a single degree $d$ with $d\in\{0,1,2,n{-}2,n{-}1,n\}$, the following conditions are equivalent.
		\begin{enumerate}
			\item[\textup{(i)}] $I$ has a linear resolution.
			\item[\textup{(ii)}] $I$ has linear quotients.
			\item[\textup{(iii)}] $I^k$ has linear quotients, for all $k\ge1$.
		\end{enumerate}
	\end{TheoremAlph}
	
	In view of the results in this paper, and the existing literature, we are tempted to conjecture that if a (squarefree) monomial ideal $I\subset S$ has linear powers, then $I^k$ has linear quotients for all $k\gg0$.
	
	\section{Dominant monomial orders and linear quotients}\label{sec2}
	
	In this section, we introduce the notions of $x$-dominant and $y$-dominant monomial orders and show that for an equigenerated monomial ideal $I$ if $\mathcal{R}(I)$ satisfies the $x$-condition with respect to either an $x$-dominant or a $y$-dominant monomial order, then all powers of $I$ possess linear quotients. We will give nice applications of these results in this and in the subsequent section, where we treat squarefree monomial ideals generated in degree $n-2$. 
	
	Let $T=K[x_1,\ldots,x_n,y_1,\ldots,y_m]$ be a polynomial ring over the field $K$. We call a monomial order $<$ on $T$ an \emph{$x$-dominant order}, if for any monomials $u,u'\in K[x_1,\ldots,x_n]$ and $v,v'\in K[y_1,\ldots,y_m]$, with $\deg(uv)=\deg(u'v')$, if $u\neq u'$ then the following is satisfied:
	$$uv>u'v'\, \Longleftrightarrow \, u>u'.$$
	
	The following monomial orders are examples of $x$-dominant orders on $T$.
	
	\begin{itemize}
		\item [(a)]  Lexicographical order induced by $x_1>\cdots>x_n>y_1>\cdots>y_m$.
		\item [(b)] Pure lexicographical order induced by $x_1>\cdots>x_n>y_1>\cdots>y_m$.
		\item [(c)] Reverse lexicographical order induced by $y_1>\cdots>y_m>x_1>\cdots>x_n$.
		\item [(d)] The product order of $<_x$ and $<_y$, where  $<_x$ and $<_y$ are given monomial orders on $K[x_1,\ldots,x_n]$ and $K[y_1,\ldots,y_m]$, respectively.
	\end{itemize}\smallskip
	
	Similarly, we call a monomial order $<$ on $T$ a \emph{$y$-dominant order}, if for any monomials $u,u'\in K[x_1,\ldots,x_n]$ and $v,v'\in K[y_1,\ldots,y_m]$, with $\deg(uv)=\deg(u'v')$, if $v\neq v'$ then the following is satisfied:
	$$uv>u'v'\, \Longleftrightarrow \, v>v'.$$
	
	The following monomial orders are examples of $y$-dominant orders on $T$.
	
	\begin{itemize}
		\item [(a)]  Lexicographical order induced by $y_1>\cdots>y_m>x_1>\cdots>x_n$.
		\item [(b)] Pure lexicographical order induced by $y_1>\cdots>y_m>x_1>\cdots>x_n$.
		\item [(c)] Reverse lexicographical order induced by $x_1>\cdots>x_n>y_1>\cdots>y_m$.
		\item [(d)] The product order of $<_y$ and $<_x$, where  $<_y$ and $<_x$ are given monomial orders on $K[y_1,\ldots,y_m]$ and $K[x_1,\ldots,x_n]$, respectively.
	\end{itemize} 
	
	\begin{Theorem}\label{GeneralStatementX}
		Let $I\subset S$ be a monomial ideal generated in a single degree. If  $\mathcal{R}(I)$ satisfies the $x$-condition with respect to an $x$-dominant monomial order, then $I^k$ has linear quotients for all $k\ge1$.    
	\end{Theorem}
	\begin{proof}
		Let $\mathcal{G}(I)=\{u_1,\ldots,u_m\}$, and let $T=K[x_1,\ldots,x_n,y_1,\dots,y_m]$ be a polynomial ring. Consider the $S$-algebra homomorphism $\varphi:T\rightarrow\mathcal{R}(I)$ defined by $\varphi(y_i)=u_it$ for $i=1,\dots,m$    and set $J=\Ker\,\varphi$. By assumption, there exists an $x$-dominant monomial order $<_d$ on $T$ such that $\mathcal{R}(I)$ satisfies the $x$-condition with respect to $<_d$. Fix an integer $k\geq 1$ and assume that $\mathcal{G}(I^k)=\{f_1,\ldots,f_r\}$ such that 
		\begin{equation}\label{OrderOnf_i}
			f_1>_d\cdots>_d f_r. 
		\end{equation}
		For any $1\leq \ell\leq r$, we set $f_{\ell}^*$ to be the standard    presentation of $f_{\ell}$ with respect to $<_d$. That is $f_{\ell}^*$ is the smallest monomial in $K[y_1,\dots,y_m]$ with respect to $<_d$ such that $\varphi(f_{\ell}^*)=f_{\ell}t^k$. Such a monomial $f_{\ell}^*$ exists, since the monomials $y_{i_1}\cdots y_{i_k}$ satisfying $\varphi(y_{i_1}\cdots y_{i_k})=f_{\ell}t^k$ are all of degree $k$ and the number of such monomials is finite.
		
		We prove that for any $2\leq j\leq r$, the ideal $L_j=(f_1,\ldots,f_{j-1}):f_j$ is generated by variables. Let $u\in L_j$ be a monomial. We need to show that there exists a variable $x_t\in L_j$ which divides $u$. Since all $f_i$'s are minimal generator of $I^k$, we have $u\neq 1$. From $u\in L_j$ we have $uf_j=vf_i$ for some monomial $v\in S$ with $v\neq u$ and an integer $i<j$. This gives the relation $h=uf_j^*-vf_i^*\in J$. 
		Since $uf_j=vf_i$ and $f_i>_d f_j$, we obtain $u=(vf_i)/f_j>_d v$. This together with the facts that $<_d$ is an $x$-dominant order and $\deg(uf_j^*)=\deg(vf_i^*)$ implies that $\ini_{<_d}(h)=uf_j^*$. Since $\mathcal{R}(I)$ satisfies the $x$-condition with respect to $<_d$, there exists a monomial $w\in K[x_1,\ldots,x_n]$ with $\deg(w)\leq 1$ and a monomial $g\in K[y_1,\ldots,y_m]$ such that $w\mid u$ and $g\mid f_j^*$ and $wg\in \ini_{<_d}(J)$ is a minimal generator. Let $h'=wg-w'g'\in J$ be a binomial with $\ini_{<_d}(h')=wg$, where $w'\in K[x_1,\ldots,x_n]$  and $g'\in K[y_1,\ldots,y_m]$ are monomials. Clearly, $\gcd(w,w')=1$ and $\deg(g)=\deg(g')$. Then comparing the degrees of monomials in the equality $w\varphi(g)=\varphi(wg)=\varphi(w'g')=w'\varphi(g')$ and using the assumption that $I$ is generated in one degree, we obtain $\deg(w)=\deg(w')$.
		
		First we show that $w\neq 1$. Suppose on the contrary that $w=1$. Since $\deg(w)=\deg(w')$, we get $w'=1$. Hence, $h'=g-g'\in J$ with $\ini_{<_d}(h')=g$. Thus  $g>_d g'$ which implies that $f_j^*=(f_j^*/g)g>_d (f_j^*/g)g'$. Moreover, $$f_j^*-(f_j^*/g)g'=(f_j^*/g)(g-g')\in J.$$ So
		$\varphi((f_j^*/g)g')=\varphi(f_j^*)=f_jt^k$. This together with $(f_j^*/g)g'<_d f_j^*$ contradicts the fact that $f_j^*$ is the standard presentation of $f_j$ with respect to $<_d$. Thus $w\neq 1$.
		
		Since $\deg(w)\leq 1$ we have $w=x_s$ for some $s$, and $w'=x_t$ for some $t$. Since $\gcd(w.w')=1$ we have $s\neq t$. From $\ini_{<_d}(h')=x_sg$, and that $<_d$ is an $x$-dominant order, we obtain $x_s>_d x_t$. Moreover, 
		\begin{equation}\label{RelationinJ}
			x_sf_j^*-x_t(f_j^*/g)g'=(f_j^*/g)[x_sg-x_tg']=(f_j^*/g) h'\in J. 
		\end{equation}  Since $I$ is generated in one degree and $(f_j^*/g)g'$ is of degree $k$, there exists an integer $1\leq p\leq r$ such that $\varphi((f_j^*/g)g')=f_pt^k=\varphi(f_p^*)$. Then from $(f_j^*/g)g'-f_p^*\in J$ and (\ref{RelationinJ}) we get the relation $x_sf_j^*-x_tf_p^*\in J$.
		This means that $x_sf_j=x_tf_p$. From $x_s>_d x_t$ we conclude that $f_p=(x_sf_j)/x_t>_d f_j$. By (\ref{OrderOnf_i}) this means that $p<j$ and $f_p\in (f_1,\ldots,f_{j-1})$. Thus the equation $x_sf_j=x_tf_p$ implies that $x_s\in L_j$, as desired.
	\end{proof}
	
	The following result is the analogue of Theorem \ref{GeneralStatementX} with respect to a $y$-dominant monomial order. Its proof follows the same strategy as that employed in the proof of \cite[Theorem 2.5]{HHDepth}. It should be mentioned that in \cite[Theorem 2.5]{HHDepth}, by the $x$-condition, the authors mean the $x$-condition with respect to the special product order given before the statement of their theorem.
	
	\begin{Theorem}\label{GeneralStatementY}
		Let $I\subset S$ be a monomial ideal generated in one degree. If  $\mathcal{R}(I)$ satisfies the $x$-condition with respect to a $y$-dominant monomial order, then $I^k$ has linear quotients for all $k\ge1$.    
	\end{Theorem}
	\begin{proof}
		Let $\mathcal{G}(I)=\{u_1,\ldots,u_m\}$, and let $T=K[x_1,\ldots,x_n,y_1,\dots,y_m]$ be a polynomial ring. Consider the $S$-algebra homomorphism $\varphi:T\rightarrow\mathcal{R}(I)$ defined by $\varphi(y_i)=u_it$ for $i=1,\dots,m$    and set $J=\Ker\,\varphi$.  Let $<_d$ be a  $y$-dominant monomial order such that $\mathcal{R}(I)$ satisfies the $x$-condition with respect to $<_d$. 
		Fix an integer $k\geq 1$ and assume that $\mathcal{G}(I^k)=\{f_1,\ldots,f_r\}$. For any $1\leq \ell\leq r$, we set $f_{\ell}^*$ to be the standard presentation of $f_{\ell}$ with respect to $<_d$ in $T$. Without loss of generality we may assume that
		\begin{equation}\label{OrderOnf*}
			f_{1}^*<_d \cdots<_d f_{r}^*.
		\end{equation}
		We show that for any $2\leq j\leq r$, the ideal $L_j=(f_1,\ldots,f_{j-1}):f_j$ is generated by variables. Let $u\in L_j$ be a monomial.   Since $f_j$ is a minimal generator of $I^k$, we have $u\neq 1$. From $u\in L_j$ we have $uf_j=vf_i$ for some monomial $v\in S$ with $v\neq u$ and an integer $i<j$. This gives the relation $h=uf_j^*-vf_i^*\in J$. 
		Note that $\ini_{<_d}(h)=uf_j^*$, since $f_j^*>_d f_i^*$, $\deg(uf_j^*)=\deg(vf_i^*)$ and $<_d$ is a $y$-dominant order. 
		Since $\mathcal{R}(I)$ satisfies the $x$-condition with respect to $<_d$, there exists a monomial $w\in K[x_1,\ldots,x_n]$ with $\deg(w)\leq 1$ and a monomial $g\in K[y_1,\ldots,y_m]$ such that $w\mid u$ and $g\mid f_j^*$ and $wg\in \ini_{<_d}(J)$ is a minimal generator.
		Let $h'=wg-w'g'\in J$ be a binomial with $\ini_{<_d}(h')=wg$, where $w'\in K[x_1,\ldots,x_n]$  and $g'\in K[y_1,\ldots,y_m]$.
		Comparing the degrees of monomials in the equality $w\varphi(g)=\varphi(wg)=\varphi(w'g')=w'\varphi(g')$ and using the assumption that $I$ is generated in one degree, we obtain $\deg(w)=\deg(w')$. Moreover, from $\ini_{<_d}(h')=wg$ we obtain $g>_d g'$. 
		
		By contradiction assume that $w=1$. Then $w'=1$ and $h'=g-g'\in J$ with $\ini_{<_d}(h')=g$. Thus  $f_j^*=(f_j^*/g)g>_d (f_j^*/g)g'$ with $$f_j^*-(f_j^*/g)g'=(f_j^*/g)(g-g')\in J.$$ These give a contradiction to $f_j^*$ being the standard presentation of $f_j$ with respect to $<_d$. 
		Thus $w\neq 1$ and we have $w=x_s$ for some $s$, and $w'=x_t$ for some $t$. Since $\gcd(w,w')=1$, we have $s\neq t$. Also
		\begin{equation}\label{Relation1}
			x_sf_j^*-x_t(f_j^*/g)g'=(f_j^*/g)[x_sg-x_tg']=(f_j^*/g) h'\in J. 
		\end{equation}  
		Since $I$ is generated in one degree, there exists an integer $p$ such that $\varphi((f_j^*/g)g')=f_pt^k=\varphi(f_p^*)$.
		Because $f_p^*$ is the standard presentation of $f_p$, we have 
		$f_p^*\leq_d (f_j^*/g)g'<_d f_j^*$. Hence, by (\ref{OrderOnf*}) we get  $p<j$. 
		Also from $(f_j^*/g)g'-f_p^*\in J$ and (\ref{Relation1}) and we get the relation $x_sf_j^*-x_tf_p^*\in J$. So $x_sf_j=x_tf_p$ with $p<j$, which implies that $x_s\in L_j$. 
	\end{proof}

	Monomial ideals with the $\ell$-exchange
	property were first considered in the work of Herzog, Hibi and Vladoiu \cite{HHV2005} in the study of ideals of fiber type. For such ideals it was shown in \cite[Theorem 5.1]{HHV2005} (see also \cite[Theorem 6.24]{EH}) that if the ideal $I$ is generated in one degree, then  $\mathcal{R}(I)$ satisfies the $x$-condition with respect to the product order obtained by the product of $<_{\lex}$ and $<$, where $<_{\lex}$ is the lexicographical order on $K[x_1,\ldots,x_n]$ induced by $x_1>\cdots>x_n$ and $<$ is an arbitrary monomial order on $K[y_1,\ldots,y_m]$. Since such order is an $x$-dominant order, Theorem \ref{GeneralStatementX} implies 
	
	\begin{Corollary}\label{Application1}
		Let $I$ be a monomial ideal generated in one degree with the $\ell$-exchange
		property. Then $I^k$ has linear quotients with respect to the lexicographical order induced by $x_1>\cdots>x_n$ for all $k\geq 1$. In particular, this is the case if $I$ is one of the following ideals. 
		\begin{enumerate}
			\item The facet ideal $I(\Gamma^{[t]})$, where $\Gamma$ is a $d$-flag sortable simplicial complex.
			\item A sortable weakly polymatroidal ideal.
			\item A generalized Hibi ideal.
		\end{enumerate} 
	\end{Corollary}
	\begin{proof}
		The ideals in (1), (2) and (3) satisfy the $\ell$-exchange
		property by \cite[Proposition 1.4]{FM2} (see also \cite[Corollary 1.6]{FM2}), \cite[Proposition 6.1]{EHM} and \cite[Theorem 2.4, Theorem 5.3, Proposition 6.1]{EHM}, where the linear resolution property for all powers of these ideals were shown.
	\end{proof}
	
	\section{Proof of Theorem \ref{ThmB}}
	
	The aim in this section is to prove Theorem \ref{ThmB}.
	First we recall some concepts which will be used in the next result. Let $G$ be a finite simple graph. A {\em walk} of length $r$ in $G$ is a subgraph $W$ of $G$ with $$E(W)\ =\ \{\{x_1,x_2\},\{x_2,x_3\},\ldots,\{x_r,x_{r+1}\}\},$$ where $x_1,\ldots,x_{r+1}$ are vertices in $G$. We denote this walk by $W:x_1,\ldots,x_{r+1}$.
	The walk $W$ is called a {\em closed walk} if $x_1=x_{r+1}$. An {\em even closed walk} is a closed walk of even length.
	For an even closed walk $W:x_1,\ldots,x_{2k},x_1$, a {\em subwalk} of $W$ is an even closed walk $W': z_1,\ldots,z_{2t},z_1$, where for each $1\leq \ell\leq t$, there exist integers $1\leq \ell',\ell''\leq k$
	such that $\{z_{2\ell-1},z_{2\ell}\}=\{x_{2\ell'-1},x_{2\ell'}\}$ and $\{z_{2\ell},z_{2\ell+1}\}=\{x_{2\ell''},x_{2\ell''+1}\}$.  
	An even closed walk $W$ is called {\em primitive}, if the only subwalk of $W$ is $W$ itself. 
	
	Let $J\subset S$ be a binomial ideal. A nonzero binomial $h=u-v\in J$ is called {\em a primitive binomial} in $J$, if there is no nonzero binomial $h'=u'-v'\in J$ with $h'\neq h$ such that $u'\mid u$ and $v'\mid v$.
	
	\begin{Theorem}\label{LinearPowers}
		Let $I=I_c(G)\subset S$, where $G$ is a connected graph. Then
		\begin{enumerate}
			\item [(a)] The Rees algebra $\mathcal{R}(I)$ satisfies the $x$-condition with respect to some lexicographic monomial order which is $x$-dominant. 
			
			\item [(b)] $I^k$ has linear quotients for all $k\ge1$. 
		\end{enumerate}
	\end{Theorem}
	\begin{proof}
		(a) We claim that there is a labeling $1,\ldots,n$ on the vertex set of $G$ such that $G_i=G\setminus\{1,\ldots,i\}$ is a connected graph. Consider a spanning tree $T$ of $G$. Then the desired labeling is obtained by labeling a leaf of $T$ by $1$, and for each $1\leq i<n-1$ labeling a leaf of $T\setminus \{1,\ldots,i\}$ by $i+1$.
		Let $E(G)=\{e_1,\ldots,e_m\}$. Say $e_i=\{r_i,s_i\}$ for $i=1,\dots,m$, and set $u_i={\bf x}_{[n]}/{\bf x}_{e_i}$. Then $\mathcal{G}(I)=\{u_1,\dots,u_m\}$.
		
		Let $T=S[y_1,\dots,y_m]$ be a polynomial ring and let $\varphi:T\rightarrow\mathcal{R}(I)$ be the $S$-algebra homomorphism defined by $\varphi(y_i)=u_it$ for $i=1,\dots,m$. Let $J=\Ker\,\varphi$. Consider the lex order  $<_\lex$ on $T$ induced by $x_1>\dots>x_n>y_1>\dots>y_m$. We prove that $\mathcal{R}(I)$ satisfies the $x$-condition with respect to this order. To this aim, consider a minimal monomial generator $uy_{i_1}\cdots y_{i_k}\in \ini_<(J)$, where $u\in S$ is a monomial. We show that $\deg(u)\leq 1$. By contradiction assume that $\deg(u)\geq 2$. Then there exists a binomial  $h=uy_{i_1}\cdots y_{i_k}-vy_{j_1}\cdots y_{j_k}\in J$ with $\ini_<(h)=uy_{i_1}\cdots y_{i_k}$, such that $v\in S$ is a monomial with $\gcd(u,v)=1$ and $u>_\lex v$. By \cite[Theorem 3.13]{HHBook2} we may assume that $h$ is a primitive binomial. Let $T'=S[z_1,\dots,z_m]$ be a polynomial ring, let  $\varphi':T'\rightarrow\mathcal{R}(I(G))$ be the $S$-algebra homomorphism defined by $\varphi'(z_i)={\bf x}_{e_i}t$ for all $i=1,\dots,m$, and let $J'=\Ker\,\varphi'$. 
		Then the relation $h\in J$ gives the relation $h'=uz_{j_1}\cdots z_{j_k}-vz_{i_1}\cdots z_{i_k}\in J'$.
		Note that one may view $\mathcal{R}(I(G))$ as the edge ring $$K[I(G^*)]=K[{\bf x}_e:\, e\in E(G^*)],$$ where $G^*$ is a graph obtained from $G$ by adding a new vertex $n+1$ to $G$ and connecting it to all  vertices of $G$. Since $h$ is a primitive binomial in $J$, it can be easily seen that $h'$ is a primitive binomial in $J'$. 
		So by \cite[Corollary 10.1.5]{HHBook}, the relation $h'$ corresponds to a primitive even closed walk in $G^*$, say $W$.   
		Since $\deg(u)\geq 2$, the vertex $n+1$ appears at least two times in $W$. Let $q_1$ be the smallest integer such that $x_{q_1}\mid u$. 
		Then we may write $$n+1,q_1,q_2,\ldots,q_d,n+1,q_{d+1},\ldots$$ as part of $W$, where $q_i$'s are vertices in $G$. Notice that $d$ is an even integer, otherwise $n+1,q_1,q_2,\ldots,q_d,n+1$ will be an even closed subwalk of $W$, which contradicts to $W$ being primitive. Moreover, we have $x_{q_1}\mid u$
		and $x_{q_d}\mid u$ with $q_1\leq q_d$ and $x_{q_{d+1}}\mid v$.
		We set $\{q_i,q_{i+1}\}=e_{s_i}$ for all $i$.  Hence,  from the relation $h'$ it follows that $$\{s_1,s_3,\ldots,s_{d-3},s_{d-1}\}\subset \{i_1,i_2,\ldots,i_k\}.$$ 
		We consider the following cases.
		
		\smallskip
		\textbf{Case 1}. Assume that $q_d<n$. Then by our labeling on $G$, the graph $G_{q_{d}-1}=G[\{q_d,q_d+1,\ldots,n\}]$ is connected and it has at least two vertices. Here, $G[A]$ denotes the induced subgraph of $G$ on the vertex set $A$. Thus there exists a vertex $\ell>q_d$ such that $e'=\{\ell,q_d\}\in E(G_{q_{d}-1})\subseteq E(G)$. Then corresponding to the even closed walk $n+1,q_1,q_2,\ldots,q_d,\ell,n+1$ we obtain the relation $$g=x_{q_1}z_{s_2}z_{s_4}\cdots z_{s_{d-2}}z_{b}-x_{\ell}z_{s_1}z_{s_3}\cdots z_{s_{d-1}}\in J',$$ where $b$ is an integer with $\varphi'(z_b)=\textbf{x}_{e'}t$.
		The relation $g$ implies that     $$f=x_{q_1}y_{s_1}y_{s_3}\cdots y_{s_{d-1}}-x_{\ell}y_{s_2}y_{s_4}\cdots y_{s_{d-2}}y_{b}\in J.$$ Since $\ell>q_d\geq q_1$, we have $\ini_<(f)=x_{q_1}y_{s_1}y_{s_3}\cdots y_{s_{d-1}}\in\ini_<(J)$ with  $x_{q_1}\mid u$ and $y_{s_1}y_{s_3}\cdots y_{s_{d-1}}\mid y_{i_1}\cdots y_{i_k}$. This contradicts to the minimality of $uy_{i_1}\cdots y_{i_k}$. 
		
		\smallskip
		\textbf{Case 2}. Assume that $q_d=n$. Then from $\gcd(u,v)=1$ and that $x_{q_{d+1}}$ divides $v$, we conclude that   $q_{d+1}\leq n-1$. Now, from the fact that $q_1$ is  the smallest integer such that $x_{q_1}$ divides $u$, that $u>_{\lex} v$ and $\gcd(u,v)=1$, we obtain that $q_1\leq n-2$. Then the graph $G_{q_1}=G[\{q_1+1,q_1+2,\ldots,n\}]$ is connected and has at least two vertices. So there exists a vertex $\ell>q_1$ such that $e'=\{\ell,q_d\}\in E(G_{q_1})\subseteq E(G)$. Similar to case 1, this leads to a relation of the form $$f=x_{q_1}y_{s_1}y_{s_3}\cdots y_{s_{d-1}}-x_{\ell}y_{s_2}z_{s_4}\cdots y_{s_{d-2}}y_{b}\in J$$ with $\ini_<(f)=x_{q_1}y_{s_1}y_{s_3}\cdots y_{s_{d-1}}$ and we get a contradiction to the minimality of $uy_{i_1}\cdots y_{i_k}$. So we conclude that $\deg(u)\leq 1$ and hence $\mathcal{R}(I)$ satisfies the $x$-condition with respect to the lex order  $<_\lex$ on $T$ induced by $x_1>\dots>x_n>y_1>\dots>y_m$.
		
		\smallskip
		(b) follows from (a) and Theorem \ref{GeneralStatementX}.
	\end{proof}
	
	Next, we establish Theorem \ref{ThmB}.
	
	\begin{Corollary}\label{Cord=n-2}
		Let $I\subset S$ be a squarefree monomial ideal generated in degree $n-2$. Write $I=I_c(G)$. The following statements are equivalent.
		\begin{enumerate}
			\item[\textup{(i)}] $I$ has a linear resolution.
			\item[\textup{(ii)}] $I^k$ has a linear resolution, for all $k\ge1$.
			\item[\textup{(iii)}] $I^k$ has linear quotients, for all $k\ge1$.
			\item[\textup{(iv)}] $c(G)=1$.
		\end{enumerate}
	\end{Corollary}
	
	\begin{proof}
		(iii) $\Rightarrow$ (ii) $\Rightarrow$ (i) is valid for any equigenerated monomial ideal.
		
		(i) $\Rightarrow$ (iv) Let $V(G)=[n]$. By the Eagon-Reiner Theorem \cite[Theorem 8.1.9]{HHBook} it follows that $S/I^\vee$ is Cohen-Macaulay, where $I^\vee=\bigcap_{e\in E(G)}P_{[n]\setminus e}$ and $P_{[n]\setminus e}=(x_i:\ i\in[n]\setminus e)$. We can write $I^\vee=I_\Delta$, where $\Delta$ is the simplicial complex on $[n]$ whose set of facets is $\mathcal{F}(\Delta)=E(G)$. Since $\Delta$ is $1$-dimensional, Reisner Theorem \cite[Theorem 8.1.6]{HHBook} implies that $S/I^\vee$ is Cohen-Macaulay, if and only if, the $0$th reduced simplicial homology group $\widetilde{H}_0(\Delta;K)$ is zero. By \cite[Problem 8.2]{HHBook}, this is the case, if and only if, $\Delta$ is connected. Hence $c(G)=1$ and statement (iv) follows.
		
		(iv) $\Rightarrow$ (iii) Since $c(G)=1$, we can write $I={\bf x}_AI_c(H)$, where $A\subset V(G)$ is the set of isolated vertices of $G$ and $H$ is the connected graph obtained from $G$ by removing the isolated vertices of $G$. Then $I^k$ has linear quotients for all $k\ge1$, if and only if, $I_c(H)^k$ has linear quotients for all $k\ge1$. Now, Theorem \ref{LinearPowers}(b) yields the assertion.
	\end{proof}
	
	We describe the order of linear quotients for $I^k$ in the following 
	
	\begin{Remark}
		Let $I\subset S$ be a squarefree monomial ideal generated in degree $n-2$ with a linear resolution. Let $G$ be the graph with $I=I_c(G)$ and $H$ be the connected graph obtained from $G$ by deleting the isolated vertices of $G$, see Corollary \ref{Cord=n-2}. It follows from the proof of Theorem \ref{LinearPowers} that for any labeling $V(H)=\{1,\ldots,s\}$ on $H$ such that $H\setminus \{1,\ldots,i\}$ is a connected graph for all $i$, the Rees algebra $\mathcal{R}(I_c(H))$ satisfies the $x$-condition with respect to the lexicographical order $<_{\lex}$ induced by $x_1>\dots>x_s>y_1>\dots>y_m$. Then by the proof of Theorem \ref{GeneralStatementX}, it follows that 
		for any integer $k\geq 1$, the ideal $I_c(H)^k$ has linear quotients with respect to the lexicographical order 
		induced by $x_1>\dots>x_s$.
		Hence $I={\bf x}_AI_c(H)^k$ has linear quotients with respect to the order $${\bf x}_Af_1,\ldots, {\bf x}_Af_r,$$ where ${\bf x}_A=x_{s+1}x_{s+2}\cdots x_n$, and $f_1>_{\lex}f_2>_{\lex}\cdots >_{\lex} f_r$ are the minimal monomial generators of $I_c(H)^k$.  
	\end{Remark}
	
	
	
	
	
	
	\section{Proofs of Theorems \ref{ThmC} and \ref{ThmD}}
	
	
	In order to prove Theorem \ref{ThmC}, we use the notion of Betti splitting.  Let $I,I_1,I_2\subset S$ be monomial ideals such that $\mathcal{G}(I)=\mathcal{G}(I_1)\sqcup\mathcal{G}(I_2)$ is the disjoint union of $\mathcal{G}(I_1)$ and $\mathcal{G}(I_2)$. We say that $I=I_1+I_2$ is a \textit{Betti splitting} (\cite{FHT}), if
	\begin{equation}\label{eq:BettiSplitEq}
		\beta_{i,j}(I)=\beta_{i,j}(I_1)+\beta_{i,j}(I_2)+\beta_{i-1,j}(I_1\cap I_2),\ \ \ \text{for all}\ i,j\ge0.
	\end{equation}
	In particular, when $I=I_1+I_2$ is a Betti splitting,  $$\reg\,I=\max\{\reg\,I_1,\,\reg\,I_2,\,\reg(I_1\cap I_2)-1\}.$$\smallskip
	By \cite[Proposition 2.1]{FHT}, $I=I_1+I_2$ is a Betti splitting, if and only if, the inclusion maps $I_1\cap I_2\rightarrow I_1$ and $I_1\cap I_2\rightarrow I_2$ are \textit{$\Tor$-vanishing}.
	
	We quote \cite[Proposition 3.1]{EK} (see also \cite[Lemma 4.2]{NV}).
	
	\begin{Lemma}\label{lemma:criteria-bs}
		Let $J,L\subset S$ be non-zero monomial ideals with $J\subset L$. Suppose there exists a map $\varphi:\mathcal{G}(J)\rightarrow\mathcal{G}(L)$ such that for any $\emptyset\ne\Omega\subseteq \mathcal{G}(J)$ we have
		$$
		\lcm(u:u\in\Omega)\in \mathfrak{m}\cdot(\lcm(\varphi(u):u\in\Omega)),
		$$
		where $\mathfrak{m}=(x_1,\dots,x_n)$. Then the inclusion map $J\rightarrow L$ is $\Tor$-vanishing.
	\end{Lemma}
	
	We are ready to prove Theorems \ref{ThmC} and \ref{ThmD}.
	
	\begin{proof}[Proof of Theorem \ref{ThmC}]
		For our convenience, we rename the variables from $x_1$ up to $x_6$ with $a,b,c,d,e,f$ and the variables $x_7,\dots,x_n$ with $y_1,\dots,y_{n-6}$. Let $J$ be the ideal in equation (\ref{eq:Te}). Terai has shown that
		\begin{equation}\label{eq:Terai}
			\reg\,J\ =\ \begin{cases}
				\,3&\textup{if}\ \textup{char}(K)\ne2,\\
				\,4&\textup{otherwise}.
			\end{cases}
		\end{equation}
		
		Let $L$ be the ideal generated by all squarefree monomials of degree 2 in the variables $\{a,b,c,d,e,f\}$. Notice that $J$ is contained in $L$, and $L$ has linear quotients and linear resolution, because it is polymatroidal. Fix $d$ with $3\le d\le n-3$, and set
		$$
		I\ =\ (\prod_{i=1}^{d-3}y_i)[J+(y_{d-2},y_{d-1},y_d,\dots,y_{n-6})L].
		$$
		
		We claim that $I$ is the desired ideal. Let $H=J+(y_{d-2},y_{d-1},y_d,\dots,y_{n-6})L$. Notice that $I$ is a squarefree, fully-supported monomial ideal generated in degree $d$ and that $\reg\,I=\reg\,H+(d-3)$. Hence, it remains to show that $\reg\,H=3$ if $\textup{char}(K)\ne 2$ and $\reg\,H=4$, otherwise. To this end, we claim that 
		\begin{equation}\label{eq:bs}
			H\ =\ J+(y_{d-2},\dots,y_{n-6})L
		\end{equation}
		is a Betti splitting. Since $J\subset L$ and the variables $y_i$ do not divide any minimal monomial generator of $J$ and $L$, we obtain that
		$$
		W\ =\ J\cap(y_{d-2},\dots,y_{n-6})L\ =\ (y_{d-2},\dots,y_{n-6})J.
		$$
		
		It order to conclude that (\ref{eq:bs}) is a Betti splitting, we must prove that the inclusion maps $W\rightarrow J$ and $W\rightarrow(y_{d-2},\dots,y_{n-6})L$ are $\Tor$-vanishing. Using Lemma \ref{lemma:criteria-bs}, it is clear that the first map is $\Tor$-vanishing. To show that the second map is $\Tor$-vanishing we prove that the map $\varphi:\mathcal{G}(W)\rightarrow\mathcal{G}(J)$, defined as follows, satisfies Lemma \ref{lemma:criteria-bs}. Each monomial $u\in\mathcal{G}(W)$ is of the form $u=y_iv$ for some $d-2\le i\le n-6$ and $v\in\mathcal{G}(J)$. Since $J\subset L$, for each $v\in\mathcal{G}(J)$ we pick a monomial $v'\in\mathcal{G}(L)$ that divides $v$. Then we set $\varphi(u)=\varphi(y_iv)=v'$. For each non-empty subset $\Omega$ of $\mathcal{G}(W)$ it is then clear that 
		$$
		\lcm(u:u\in\Omega)\ \in\ (y_{d-2},\dots,y_{n-6})\lcm(\varphi(u):u\in\Omega)\subset\m(\lcm(\varphi(u):u\in\Omega)),
		$$
		because $\lcm(u:u\in\Omega)$ is divided by $y_i$ for some $i$ with $d-2\le i\le n-6$, while $\lcm(\varphi(u):u\in\Omega)$ is a monomial in the variables $a,b,c,d,e,f$. Hence, (\ref{eq:bs}) is indeed a Betti splitting.
		
		Since the supports of $(y_{d-2},\dots,y_{n-6})$ and $L$ are disjoint, and both ideals have linear resolutions, it follows that $\reg\,(y_{d-2},\dots,y_{n-6})L=3$. Similarly, for the same reason, we have that $\reg\,W=\reg\,J+1$. Consequently,
		$$
		\reg\,H\ =\ \max\{\reg\,J,\,\reg\,(y_{d-2},\dots,y_{n-6})L,\,\reg\,W-1\}\ =\ \max\{\reg\,J,\,3\}.
		$$
		Now, formula (\ref{eq:Terai}) yields the assertion.
	\end{proof}
	
	\begin{proof}[Proof of Theorem \ref{ThmD}]
		As before, we rename the variables from $x_1$ up to $x_6$ with $a,b,c,d,e,f$ and the variables $x_7,\dots,x_n$ with $y_1,\dots,y_{n-6}$. Let $J$ be the ideal in equation (\ref{eq:St}). Sturmfels \cite{S} has shown that $J$ has linear quotients with respect to its minimal generators in the given order, but $J^2$ does not have a linear resolution.
		
		Now, let $L$ be the ideal generated by all squarefree monomials of degree 2 in the variables $\{a,b,c,d,e,f\}$. Notice that $L$ has linear quotients, hence a linear resolution, and $J$ is contained in $L$. Fix $d$ with $3\le d\le n-3$, and set
		$$
		I\ =\ (\prod_{i=1}^{d-3}y_i)[J+(y_{d-2},y_{d-1},y_d,\dots,y_{n-6})L].
		$$
		
		We claim that $I$ is the desired ideal. Let $H=J+(y_{d-2},y_{d-1},y_d,\dots,y_{n-6})L$. Notice that $I$ is a squarefree, fully-supported monomial ideal generated in degree $d$ and $I$ has linear quotients if and only if $H$ has linear quotients. By \cite[Lemma 4.7(b)]{FM}, $(y_{d-2},\dots,y_{n-6})L$ has linear quotients, say with linear quotients order $u_1,\dots,u_r$. Let $v_1=def$, $v_2=cef$, $v_3=cdf$, $v_4=cde$, $v_5=bef$, $v_6=bcd$, $v_7=acf$, $v_8=ade$. As remarked before, $v_1,\dots,v_8$ is a linear quotients order of $J$. We claim that
		$$
		u_1,\dots,u_r,v_1,\dots,v_8
		$$
		is a linear quotients order of $H$. For all $i=2,\dots,r$, the ideal $(u_1,\dots,u_{i-1}):u_i$ is generated by variables. Let $1\le i\le 8$, we claim that $(u_1,\dots,u_r,v_1,\dots,v_{i-1}):v_i$ is also generated by variables. Since $(v_1,\dots,v_{i-1}):v_i$ is generated by variables, it is enough to show that $(u_1,\dots,u_r):v_i$ is also generated by variables. Since $v_i\in J\subset L$ and the variables $y_j$ do not divide $v_i$ nor any minimal generator of $L$, it follows that
		$$
		(u_1,\dots,u_r):v_i=(y_{d-2},\dots,y_{n-6})(L:v_i)=(y_{d-2},\dots,y_{n-6}),
		$$
		as desired. This shows that $H$, and therefore $I$, has indeed linear quotients.
		
		Finally, we claim that $I^2$ does not have a linear resolution. Indeed, the restriction (see \cite[Lemma 4.4]{HHZ2}) of the ideal
		$$
		I^2\ =\ (\prod_{i=1}^{d-3}y_i)^2J^2+(\prod_{i=1}^{d-3}y_i)(y_{d-2},\dots,y_{n-6})JL+(y_{d-2},\dots,y_{n-6})^2L^2
		$$
		with respect to the set of variables $\{a,b,c,d,e,f,y_1,\dots,y_{d-3}\}$ is the ideal $(\prod_{i=1}^{d-3}y_i)J^2$.
		
		Since $J^2$ does not have a linear resolution, $(\prod_{i=1}^{d-3}y_i)J^2$ does not have a linear resolution. By \cite[Lemma 4.4]{HHZ2}, $I^2$ can not have a linear resolution, as desired.
	\end{proof}
	
	As a consequence of Theorem \ref{ThmE} we obtain
	
	\begin{Corollary}
		Let $I\subset S$ be a squarefree monomial ideal with $|\supp(I)|\leq 5$. Then 
		\begin{enumerate}
			\item[\textup{(a)}] The linear resolution property for $I$ does not depend on the field $K$.
			\item[\textup{(b)}] The linear powers property for  $I$ does not depend on the field $K$.
		\end{enumerate}
		Moreover, $I$ has a linear resolution if and only if $I^k$ has linear quotients for all $k$. 
		
		In particular, a squarefree monomial ideal $I$ for which $I$ has a linear resolution and some power of $I$ does not have a linear resolution should satisfy $|\supp(I)|\geq 6$.
	\end{Corollary}

	\noindent\textbf{Acknowledgment.}
	A. Ficarra was partly supported by INDAM (Istituto Nazionale di Alta Matematica), and also by the Grant JDC2023-051705-I funded by MICIU/AEI/10.13039/501100011033 and by the FSE+.

\end{document}